\documentclass[11pt,reqno]{amsart}
 \usepackage{amssymb,amsfonts,amscd,amsbsy, color, amsmath}
\usepackage[mathscr]{eucal}
\usepackage{url}
\usepackage{graphicx}
\usepackage{mathtools}

\newtheorem{theorem}{Theorem}[section]
\newtheorem{lemma}[theorem]{Lemma}

\newtheorem{corollary}[theorem]{Corollary}
\theoremstyle{definition}

\numberwithin{equation}{section}

\newcommand{\sgn}{\operatorname{sgn}}

\makeatletter
\def\imod#1{\allowbreak\mkern5mu({\operator@font mod}\,\,#1)}
\makeatother
\allowdisplaybreaks

\begin{document}

\title[The colored Jones polynomial for double twist knots]{The colored Jones polynomial and Kontsevich-Zagier series for double twist knots}

\author{Jeremy Lovejoy}

\author{Robert Osburn}

\address{CNRS, Universit{\'e} de Paris, Case 7014, 75205 Paris Cedex 13, FRANCE}

\address{School of Mathematics and Statistics, University College Dublin, Belfield, Dublin 4, Ireland}

\email{lovejoy@math.cnrs.fr}

\email{robert.osburn@ucd.ie}


\dedicatory{In memory of Toshie Takata}

\date{\today}

\begin{abstract}
Using a result of Takata, we prove a formula for the colored Jones polynomial of the double twist knots $K_{(-m,-p)}$ and $K_{(-m,p)}$ where $m$ and $p$ are positive integers. In the $(-m,-p)$ case, this leads to new families of $q$-hypergeometric series generalizing the Kontsevich-Zagier series. Comparing with the cyclotomic expansion of the colored Jones polynomials of $K_{(m,p)}$ gives a generalization of a duality at roots of unity between the Kontsevich-Zagier function and the generating function for strongly unimodal sequences.     
\end{abstract}

\maketitle

\section{Introduction}
Let $K$ be a knot and $J_{N}(K; q)$ be the usual $N$th colored Jones polynomial, normalized to be $1$ for the unknot. It is often useful to determine a $q$-hypergeometric expression for $J_{N}(K; q)$, and such expressions have been computed for a variety of knots (e.g., \cite{habiro, hikami1, hikami2, thang, masbaum, walsh}).  These have applications to the AJ conjecture \cite{hikami1}, the generalized Volume conjecture \cite{gukov, hikami3}, the computation of WRT invariants \cite{bhl, hikami4} and quantum modular and mock modular forms \cite{folsom, go, hl1, hl2, zagier}.  

Consider the family of double twist knots $K_{(m,p)}$, where $2m$ and $2p$ are nonzero integers denoting the number of half-twists in each respective region of Figure \ref{fig:dt}. Positive integers correspond to right-handed half-twists and negative integers correspond to left-handed half-twists.

\begin{figure}[h!]
\includegraphics[width=7.5cm, height=5.0cm]{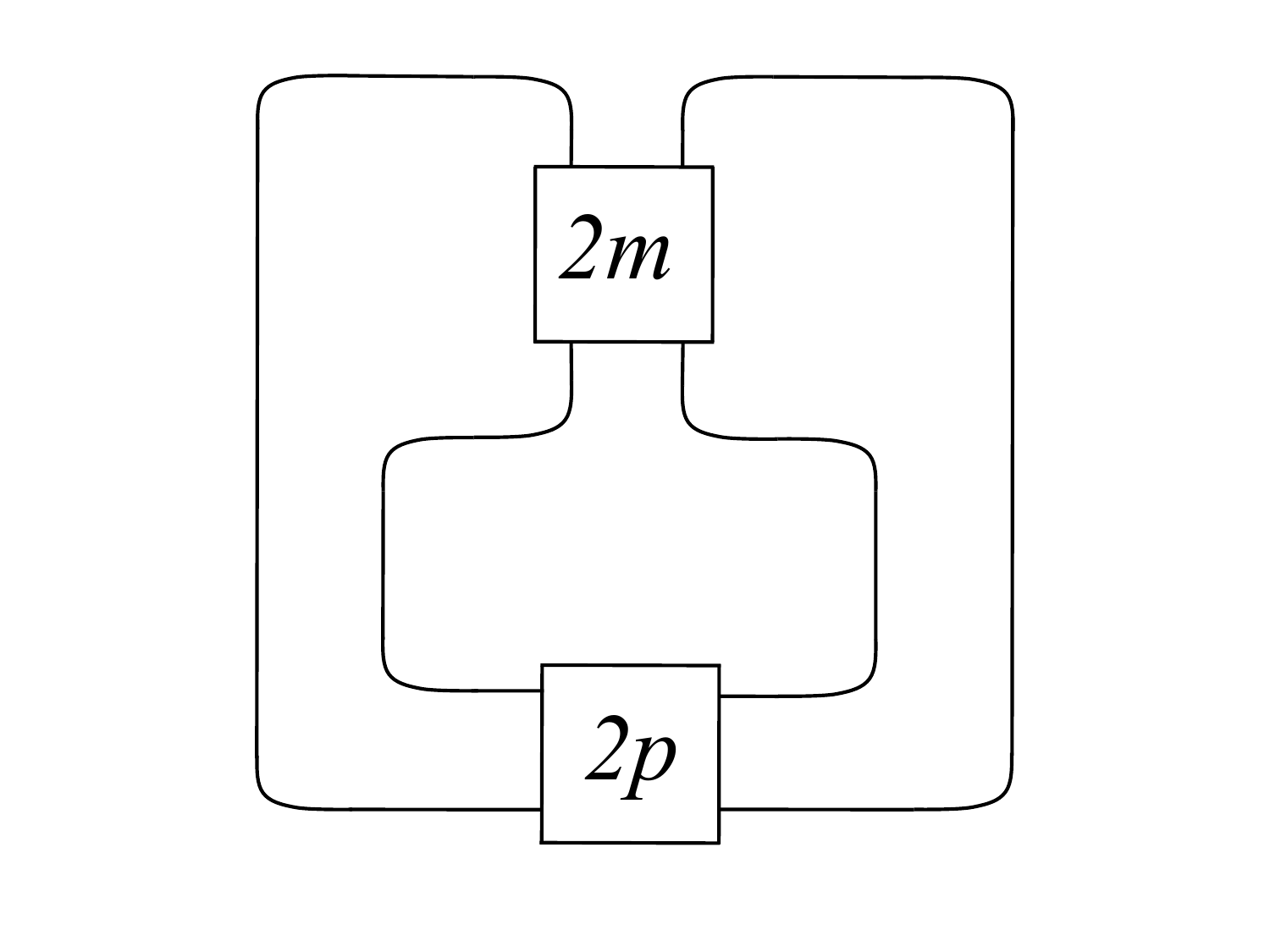}
\caption{Double twist knots}
\label{fig:dt}
\end{figure}

Using ideas of Habiro \cite{habiro} and Masbaum \cite{masbaum}, Lauridsen \cite[Section 6]{La} gave the \emph{cyclotomic expansion} of $J_N(K_{(m,p)};q)$ for all nonzero $m$ and $p$ (cf. \cite{gk,walsh}). For example, when $m$ and $p$ are positive integers we have
\begin{equation} \label{cyclotomicmp}
J_N(K_{(m,p)}; q) = \sum_{\substack{n \geq 0 \\ n =n_m \geq n_{m-1} \geq \cdots \geq n_1 \geq 0 \\ n = s_p \geq s_{p-1} \geq \cdots \geq s_1 \geq 0}}
(q^{1+N})_{n} (q^{1-N})_n q^n \prod_{i=1}^{m-1} q^{n_i^2+n_i} \begin{bmatrix} n_{i+1} \\ n_i \end{bmatrix} \prod_{j=1}^{p-1} q^{s_j^2+s_j} \begin{bmatrix} s_{j+1} \\ s_j \end{bmatrix}
\end{equation}
and 
\begin{equation} \label{cyclotomicm-p}
\begin{aligned}
& J_N(K_{(m,-p)}; q) \\
& = \sum_{\substack{n \geq 0 \\ n = n_m \geq n_{m-1} \geq \cdots \geq n_1 \geq 0 \\ n = s_p \geq s_{p-1} \geq \cdots \geq s_1 \geq 0}}
(q^{1+N})_{n} (q^{1-N})_n (-1)^nq^{-\binom{n+1}{2}} \prod_{i=1}^{m-1} q^{n_i^2+n_i} \begin{bmatrix} n_{i+1} \\ n_i \end{bmatrix} \prod_{j=1}^{p-1} q^{-s_j - s_{j+1}s_j} \begin{bmatrix} s_{j+1} \\ s_j \end{bmatrix}.
\end{aligned}
\end{equation}
\noindent Here we have used the usual $q$-binomial coefficient

\begin{equation}
\begin{bmatrix} n \\ k \end{bmatrix} = \begin{bmatrix} n \\ k \end{bmatrix}_{q} := \frac{(q)_n}{(q)_{n-k} (q)_k}
\end{equation}
with the standard $q$-hypergeometric notation

\begin{equation*}
(a)_n = (a;q)_n := \prod_{k=0}^{n-1} (1-aq^{k}).
\end{equation*}

\noindent Recall that
\begin{equation} \label{mirror}
J_{N}(K; q^{-1}) = J_{N}(K^{*}; q),
\end{equation}
where $K^*$ denotes the mirror image of the knot $K$.   Thus, since $K_{(m,p)} = K_{(p,m)}$ and $K_{(-m,-p)}$ is the mirror image of $K_{(m,p)}$, equations \eqref{cyclotomicmp} and \eqref{cyclotomicm-p} cover all of the double twist knots, up to a substitution of $q$ by $q^{-1}$.   

In this paper, we use a result of Takata \cite{takata} to prove $q$-hypergeometric formulas for the colored Jones polynomial of the double twist knots $K_{(-m,-p)}$ and $K_{(-m,p)}$ which are different from those corresponding to \eqref{cyclotomicmp} and \eqref{cyclotomicm-p}.    To state the first case, define the functions $\epsilon_{i,j,m}$ and $\gamma_{i,m}$ by
\begin{equation} \label{epsilondef}
\epsilon_{i,j,m} = 
\begin{cases}
1, &\text{if $j \equiv -i$ or $-i-1$ $\pmod{2m}$}, \\
-1, &\text{if $j \equiv i$ or $i-1$ $\pmod{2m}$}, \\
0, &\text{otherwise}
\end{cases}
\end{equation}
where $1 \leq i < j \leq 2mp-1$ with $m \nmid i$ and
\begin{equation} \label{gammadef}
\gamma_{i,m} =
\begin{cases}
1, &\text{if $i \equiv 1,\dots, m-1 \pmod{2m}$}, \\
-1 &\text{otherwise}
\end{cases}
\end{equation}
where $1 \leq i \leq 2mp-2$. Our first main result is the following.

\begin{theorem} \label{main1}
For positive integers $m$ and $p$, we have
\begin{equation} \label{JNK-m-p}
\begin{aligned}
J&_{N}(K_{(-m,-p)}; q) \\ 
&= q^{1-N}\sum_{N-1 \geq n_{2mp-1} \geq \cdots \geq n_1 \geq 0} (q^{1-N})_{n_{2mp-1}} (-1)^{n_{2mp-1}} q^{-\binom{n_{2mp-1} + 1}{2}} \prod_{\substack{1 \leq i < j \leq 2mp-1 \\ m \nmid i}} q^{\epsilon_{i,j,m}n_in_j} \\
&  \qquad \qquad \qquad \qquad \times \prod_{i=1}^{2p-1} (-1)^{n_{mi}} q^{(-1)^{i} N n_{mi} + \binom{n_{mi} + 1}{2}} \prod_{i=1}^{2mp - 2} q^{-n_i n_{i+1} + \gamma_{i, m} n_i}\begin{bmatrix} n_{i+1} \\ n_{i} \end{bmatrix}. 
\end{aligned}
\end{equation}
\end{theorem}

The case $m=1$ of Theorem \ref{main1} was proved by Takata \cite{takata}\footnote{Note that Takata's $q$-hypergeometric notation differs slightly from the standard one we use here}.   Here $K_{(1,p)}=K_{p>0}$, the usual $p$th twist knot.  As $K_{(-1,-p)}$ is the mirror image of $K_{(1,p)}$, one recovers
$J_{N}(K_{p>0}^{*}; q)$ by taking $m=1$ in \eqref{JNK-m-p} and simplifying. Namely, we have (cf. equation (15) in \cite{takata})

\begin{equation} \label{kppos}
\begin{aligned}
J_{N}&(K_{p>0}^{*}; q) \\
& = q^{1-N}\sum_{N-1 \geq n_{2p-1} \geq \cdots \geq n_1 \geq 0} (q^{1-N})_{n_{2p-1}}  q^{-Nn_{2p-1}} \prod_{i=1}^{2p-2} (-1)^{n_{i}} q^{(-1)^{i} N n_{i} + \binom{n_{i}}{2} -n_i n_{i+1}} \begin{bmatrix} n_{i+1} \\ n_{i} \end{bmatrix}. 
\end{aligned}
\end{equation}

\noindent For an example with $m > 1$, take $m=p=2$.  Then we have
\begin{equation}
\begin{aligned}
J_N&(K_{(-2,-2)};q) \\
 &= q^{1-N} \sum_{N-1 \geq n_7 \geq n_6 \geq n_5 \geq n_4 \geq n_3 \geq n_2 \geq n_1 \geq 0}  (q^{1-N})_{n_7} (-1)^{n_2+n_4+n_6+n_7}q^{N(-n_2+n_4-n_6)} \\
&\qquad \qquad \qquad \times q^{-\binom{n_7+1}{2} + \binom{n_6+1}{2} + \binom{n_4+1}{2} + \binom{n_2+1}{2} +n_1(n_3-n_4-n_5+n_6+n_7) + n_3(n_5-n_6-n_7) +  n_5n_7} \\
& \qquad \qquad \qquad \qquad \times q^{-n_2n_3 -n_4n_5 - n_6n_7 + n_1 - n_2 -n_3 - n_4 + n_5 - n_6}\begin{bmatrix} n_7 \\ n_6 \end{bmatrix} \begin{bmatrix} n_6 \\ n_5 \end{bmatrix} \begin{bmatrix} n_5 \\ n_4 \end{bmatrix} \begin{bmatrix} n_4 \\ n_3\end{bmatrix}  \begin{bmatrix} n_3 \\ n_2 \end{bmatrix} \begin{bmatrix} n_2 \\ n_1 \end{bmatrix}.
\end{aligned}
\end{equation}

For the case of $K_{(-m,p)}$, define the functions $\Delta_{i,j,m}$ and $\beta_{i,m}$ by
\begin{equation} \label{Deltadef}
\Delta_{i,j,m} = 
\begin{cases}
1, &\text{if $j \equiv -i$ or $-i+1$ $\pmod{2m}$}, \\
-1, &\text{if $j \equiv i$ or $i+1$ $\pmod{2m}$}, \\
0, &\text{otherwise}
\end{cases}
\end{equation}
where $1 \leq i < j \leq 2mp$ with $m \nmid i$ and
\begin{equation} \label{betadef}
\beta_{i,m} =
\begin{cases}
1, &\text{if $i \equiv 1,\dots, m-1 \pmod{2m}$}, \\
-1 &\text{if $i \equiv m+1,\dots, 2m-1 \pmod{2m}$}, \\
0, &\text{otherwise} 
\end{cases}
\end{equation}
where $1 \leq i \leq 2mp-1$. Our second main result is the following.

\begin{theorem} \label{main2}
For positive integers $m$ and $p$, we have
\begin{equation} \label{JNK-mp}
\begin{aligned}
J_{N}&(K_{(-m,p)}; q) \\ 
&= \sum_{N-1 \geq n_{2mp} \geq \cdots \geq n_1 \geq 0} (q^{1-N})_{n_{2mp}}(-1)^{n_{2mp}} q^{-\binom{n_{2mp} + 1}{2}} \prod_{\substack{1 \leq i < j \leq 2mp \\ m \nmid i}} q^{\Delta_{i,j,m} n_i n_j} \\
& \qquad \qquad \qquad \qquad \times \prod_{i=1}^{2p-1} (-1)^{n_{mi}} q^{(-1)^{i+1} N n_{mi} + \binom{n_{mi} + 1}{2}} \prod_{i=1}^{2mp - 1} q^{\beta_{i,m} n_i}\begin{bmatrix} n_{i+1} \\ n_{i} \end{bmatrix}.
\end{aligned}
\end{equation}
\end{theorem}

Again, the case $m=1$ was established by Takata \cite{takata}.    As $K_{(1,-p)} = K_{p<0}$ and $K_{(-1,p)}$ is the mirror image of $K_{(1,-p)}$, one recovers $J_{N}(K_{p<0}^{*}; q)$ by taking $m=1$ in (\ref{JNK-mp}) and simplifying.  Namely, we have (cf. equation (14) in \cite{takata})

\begin{equation} \label{kpneg}
\begin{aligned}
J_{N}&(K_{p<0}^{*}; q) \\
& = \sum_{N-1 \geq n_{2p} \geq \cdots \geq n_1 \geq 0} (q^{1-N})_{n_{2p}} (-1)^{n_{2p}} q^{-\binom{n_{2p} + 1}{2}} \prod_{i=1}^{2p-1} (-1)^{n_{i}} q^{(-1)^{i+1} N n_{i} + \binom{n_{i} + 1}{2}} \begin{bmatrix} n_{i+1} \\ n_{i} \end{bmatrix}. 
\end{aligned}
\end{equation}

\noindent For an example with $m > 1$, take $m=3$ and $p=1$.    Then we have
\begin{equation}
\begin{aligned}
J_N&(K_{(-3,1)};q)  \\
&= \sum_{N-1 \geq n_6 \geq n_5 \geq n_4 \geq n_3 \geq n_2 \geq n_1 \geq 0} (q^{1-N})_{n_6}(-1)^{n_3+n_6}q^{Nn_3 -\binom{n_6+1}{2} + \binom{n_3+1}{2}} \\
&\qquad \qquad \qquad \qquad \times q^{n_1(-n_2 + n_5+n_6) + n_2(-n_3+n_4+n_5)  - n_4n_5 - n_5n_6} \\
& \qquad \qquad \qquad \qquad \qquad \times q^{n_1 +n_2 - n_4 -n_5} \begin{bmatrix} n_6 \\ n_5 \end{bmatrix} \begin{bmatrix} n_5 \\ n_4 \end{bmatrix} \begin{bmatrix} n_4 \\ n_3\end{bmatrix}  \begin{bmatrix} n_3 \\ n_2 \end{bmatrix} \begin{bmatrix} n_2 \\ n_1 \end{bmatrix}.
\end{aligned}
\end{equation}

Motivated by the expressions in (\ref{JNK-m-p}) and \eqref{cyclotomicmp}, we define the $q$-series $F_{m,p}(q)$ and $U_{m,p}(x;q)$ by

\begin{equation} \label{Fmp}
\begin{aligned}
F_{m,p}(q) &= q\sum_{n_{2mp-1} \geq \cdots \geq n_1 \geq 0} (q)_{n_{2mp-1}} (-1)^{n_{2mp-1}} q^{-\binom{n_{2mp-1} + 1}{2}} \prod_{\substack{1 \leq i < j \leq 2mp-1 \\ m \nmid i}} q^{\epsilon_{i,j,m}n_in_j} \\
&  \qquad \qquad \qquad \qquad \times \prod_{i=1}^{2p-1} (-1)^{n_{mi}} q^{\binom{n_{mi} + 1}{2}} \prod_{i=1}^{2mp - 2} q^{-n_i n_{i+1} + \gamma_{i, m} n_i}\begin{bmatrix} n_{i+1} \\ n_{i} \end{bmatrix}
\end{aligned}
\end{equation}
and
\begin{equation} \label{Ump}
U_{m,p}(x; q) = \sum_{\substack{n \geq 0 \\ n =n_m \geq n_{m-1} \geq \cdots \geq n_1 \geq 0 \\ n =s_p \geq s_{p-1} \geq \cdots \geq s_1 \geq 0}}
(-xq)_{n} (-x^{-1}q)_n q^n \prod_{i=1}^{m-1} q^{n_i^2+n_i} \begin{bmatrix} n_{i+1} \\ n_i \end{bmatrix} \prod_{j=1}^{p-1} q^{s_j^2+s_j} \begin{bmatrix} s_{j+1} \\ s_j \end{bmatrix}.
\end{equation}

\noindent  Note that if $\zeta_N$ is \emph{any} $N$th root of unity, we have 
\begin{equation} \label{Fatroot}
F_{m,p}(\zeta_N)  = J_N(K_{(-m,-p)};\zeta_N)
\end{equation}
and
\begin{equation} \label{Uatroot}
U_{m,p}(-1; \zeta_N) = J_N(K_{(m,p)};\zeta_N).
\end{equation}

The base case $F_{1,1}(q)$ is $q$ times the Kontsevich-Zagier series, 
\begin{equation} \
F(q) := \sum_{n \geq 0} (q)_n,
\end{equation}
one of the foundational examples of Zagier's quantum modular forms \cite{z-1,zagier}, while
\begin{equation}
U_{1,1}(x;q) = q^{-1}U(x;q),
\end{equation}
where  $U(x;q)$ is the generating function for strongly unimodal sequences \cite{bopr},
\begin{equation} \label{u}
U(x;q) :=\sum_{n=0}^{\infty} (-xq)_n (-x^{-1}q)_n q^{n+1}.
\end{equation}
For this reason, we call the series $F_{m,p}(q)$ the \emph{generalized Kontsevich-Zagier functions for double twist knots} and the $U_{m,p}(x;q)$ the \emph{generalized $U$-functions for double twist knots}. Note that while $U_{m,p}(x;q)$ is well-defined  for $|q| <1$ and for $q$ a root of unity when $x=-1$, the generalized Kontsevich-Zagier functions $F_{m,p}(q)$ are only defined at roots of unity. 

The original Kontsevich-Zagier series $F(q)$ and the generating function for strongly unimodal sequences $U(x;q)$ when $x=-1$ are \emph{dual} at roots of unity via
\begin{equation} \label{fu}
F(\zeta_N^{-1}) = U(-1; \zeta_N).
\end{equation}
This was first shown by Bryson, Ono, Pitman and Rhoades \cite{bopr}.   It also
follows at once from \eqref{mirror} and the case $m=p=1$ of \eqref{Fatroot} and \eqref{Uatroot}, which was first observed in \cite{hl1}. Using \eqref{mirror} and the general case of \eqref{Fatroot} and \eqref{Uatroot}, we immediately have the following generalization of \eqref{fu}.

\begin{corollary} \label{fugeneralthm} If $\zeta_N$ is any $N$th root of unity, then we have
\begin{equation} \label{fugeneraleq}
F_{m,p}(\zeta_N) = U_{m,p}(-1; \zeta_N^{-1}).
\end{equation} 
\end{corollary}

Next we turn our attention to \eqref{JNK-mp} and \eqref{cyclotomicm-p}.   Motivated by these expressions, we define $\mathcal{F}_{m,p}(q)$ and $\mathcal{U}_{m,p}(x; q)$ by
\begin{equation} \label{calFmp}
\begin{aligned}
\mathcal{F}_{m,p} (q)
&= \sum_{n_{2mp} \geq \cdots \geq n_1 \geq 0} (q)_{n_{2mp}}(-1)^{n_{2mp}} q^{-\binom{n_{2mp} + 1}{2}} \prod_{\substack{1 \leq i < j \leq 2mp \\ m \nmid i}} q^{\Delta_{i,j,m} n_i n_j} \\
& \qquad \qquad \qquad \qquad \times \prod_{i=1}^{2p-1} (-1)^{n_{mi}} q^{\binom{n_{mi} + 1}{2}} \prod_{i=1}^{2mp - 1} q^{\beta_{i,m} n_i}\begin{bmatrix} n_{i+1} \\ n_{i} \end{bmatrix}
\end{aligned}
\end{equation}
and
\begin{equation} \label{calUmp}
\begin{aligned}
& \mathcal{U}_{m,p}(x; q) \\
& = \sum_{\substack{n \geq 0 \\ n = n_m \geq n_{m-1} \geq \cdots \geq n_1 \geq 0 \\ n =s_p \geq s_{p-1} \geq \cdots \geq s_1 \geq 0}}
(-xq)_{n} (-x^{-1}q)_n (-1)^n q^{-\binom{n+1}{2}} \prod_{i=1}^{m-1} q^{n_i^2+n_i} \begin{bmatrix} n_{i+1} \\ n_i \end{bmatrix} \prod_{j=1}^{p-1} q^{-s_j - s_{j+1}s_j} \begin{bmatrix} s_{j+1} \\ s_j \end{bmatrix}.
\end{aligned}
\end{equation}
 Neither $\mathcal{F}_{m,p}(q)$ nor $\mathcal{U}_{m,p}(x; q)$ is defined anywhere except at roots of unity. Note that
\begin{equation} \label{calFatroot}
\mathcal{F}_{m,p}(\zeta_N)  = J_N(K_{(-m,p)};\zeta_N) 
\end{equation}
and
\begin{equation} \label{calUatroot}
\mathcal{U}_{m,p}(-1; \zeta_N) = J_N(K_{(m,-p)};\zeta_N).
\end{equation}
As a result of \eqref{mirror}, \eqref{calFatroot} and \eqref{calUatroot}, we immediately obtain the following duality at roots of unity.

\begin{corollary} \label{calFcalU} If $\zeta_N$ is any $N$th root of unity, then we have
\begin{equation}
\mathcal{F}_{m,p}(\zeta_N) = \mathcal{U}_{m,p}(-1; \zeta_N^{-1}).
\end{equation}
\end{corollary}

We expect (\ref{JNK-m-p}), (\ref{JNK-mp}) and Theorems 1.1 and 1.2 in \cite{lo2} to be of considerable interest to those working at the overlap of number theory, quantum topology, combinatorics and physics. Recently, Park \cite{park} used (\ref{JNK-m-p}) and (\ref{JNK-mp}) to check conjectures of Gukov and Manolescu \cite{gm} concerning the existence of a two-variable invariant (with origins in string theory) for knot complements. In \cite{lovejoyqqid}, the first author initiated the study of quantum $q$-series identities with Corollary \ref{fugeneralthm} as a key illustration. Furthermore, we suspect that the generalized Kontsevich-Zagier series $F_{m,p}(q)$ given by (\ref{Fmp}) are quantum modular forms. 

The paper is organized as follows. In Section 2, we recall Takata's main theorem and provide some preliminaries. In Section 3, we prove Theorems \ref{main1} and \ref{main2}.  We conclude in Section 4 with a discussion of the generalized Kontsevich-Zagier functions and the generalized $U$-functions for the torus knots $T_{(2,2t+1)}$, along with some related questions.

\section{Preliminaries}

We begin by recalling the setup from \cite{takata}. Let $l$ and $t$ be coprime odd integers with $l > t \geq 1$ and $p^{\prime} := \frac{l-1}{2}$. For $1 \leq j \leq p^{\prime}$, define integers $r(j)$ such that $r(j) \equiv (2j-1)t \pmod{2l}$ and $-l < r(j) < l$. We put $\sigma_j := (-1)^{\lfloor \frac{(2j-1)t}{l}\rfloor}$, $r^{\prime}(j) := \frac{\mid r(j) \mid + 1}{2}$ and $i_{r^{\prime}(j)}=j$ (and thus $i_k = j$ if and only if $r^{\prime}(j)=k$). For an integer $i$, sgn($i$) denotes the sign of $i$. Let $\underline{n} = (n_1, \dotsc, n_{p^{\prime}})$ and $n_s=0$ for $s \leq 0$. Finally, define

\begin{equation} \label{kappa}
\kappa(p^{\prime})= \left\{
\begin{array}{ll}
-Nn_{p^{\prime}} & \text{if $\sigma_{p^{\prime}}=-1$,} \\
0 & \text{if $\sigma_{p^{\prime}}=1$} \\
\end{array} \right. \\
\end{equation}

\noindent and

\begin{equation} \label{tau}
\tau(j) = \left\{
\begin{array}{ll}
(-1)^{n_j - n_{j-1}} & \text{if $\sigma_j=-1$,} \\
q^{\binom{n_j - n_{j-1} + 1}{2}} & \text{if $\sigma_{j}=1$.} \\
\end{array} \right. \\
\end{equation}

Consider the family of 2-bridge knots $\frak{b}(l,t)$ \cite{bz}. The main result in \cite{takata} is an explicit formula for the colored Jones polynomial of $\frak{b}(l,t)^{*}$.

\begin{theorem} \label{takata} We have
\begin{equation} \label{spt}
J_N(\frak{b}(l,t)^{*}; q) = \sum_{N-1 \geq n_{p^{\prime}} \geq \dotsc \geq n_1 \geq 0} q^{a({\underline{n}})N + b_1(\underline{n}) + b_2(\underline{n})} X(\underline{n})
\end{equation}
where\footnote{Note that there is a misprint in the definition of $X(\underline{n})$ in \cite{takata}.   Each $\overline{q}$ in the prefactor should be $q$.}
\begin{equation*}
\begin{aligned}
a({\underline{n}}) & = -\frac{1}{2} \sum_{j=1}^{p^{\prime}} \Biggl( \sum_{k=r^{\prime}(j)}^{p^{\prime}} (\sigma_{i_k} + \sigma_{i_{p^{\prime} + 1 -k}})\Biggr) (n_j - n_{j-1}) - \frac{1}{2} \sum_{j=1}^{p^{\prime} - 1} (\sigma_{j+1} + \sigma_{p^{\prime} + 1 - j}) n_j \\
& - \frac{1}{2} (\sigma_{p^{\prime}} + 1) n_{p^{\prime}} - \sum_{j=1}^{p^{\prime}} \sigma_j, \\
b_1(\underline{n}) & = - a({\underline{n}}) + \sum_{k=1}^{\frac{l-t}{2}} \frac{1-\sigma_{i_k}}{2} n_{i_k - 1} - \sum_{k=\frac{l-t}{2}+1}^{p^{\prime}} n_{i_k - 1} + \sum_{k=\frac{l-t}{2}+1}^{p^{\prime}} \frac{1+\sigma_{i_k}}{2} n_{i_k} - (1+\sigma_{p^{\prime}}) n_{p^{\prime}} \\
& + \frac{1}{2} \sum_{j=1}^{p^{\prime} - 1} (\sigma_{j+1} - \sigma_j) n_j \\
& - \frac{1}{2} \sum_{k=1}^{p^{\prime} - 1} \sum_{k^{\prime} = k+1}^{p^{\prime}} \frac{1 + \sgn(i_k - i_{k^{\prime}})}{2}(\sigma_{i_k} - \sigma_{i_{k^{\prime}}})(n_{i_k} - n_{i_k - 1})(n_{i_{k^{\prime}}} - n_{i_{k^{\prime}} - 1}) \\
& + \sum_{j=1}^{p^{\prime}} \sigma_j \Bigl( \sum_{k=1}^{r^{\prime}(j)} (n_{i_k} - n_{i_k - 1}) \Bigr) n_{j-1}, \\
b_2(\underline{n}) & = \left\{
\begin{array}{ll}
\displaystyle \sum_{k=\frac{l-t}{2} + 1}^{\frac{t-1}{2}} \frac{1+ \sigma_{i_k}}{2} n_{i_k - 1} & \text{if $l < 2t$} \\
\displaystyle - \sum_{k=\frac{t+1}{2} + 1}^{\frac{l-t}{2}} \frac{1+ \sigma_{i_k}}{2} n_{i_k - 1} & \text{if $l > 2t,$} \\
\end{array} \right. \\
X(\underline{n}) & = (-1)^{n_{p^{\prime}}} q^{\kappa(p^{\prime})} \frac{(q)_{N-1} (q)_{n_{p^{\prime}}}}{(q)_{N-n_{p^{\prime}} - 1}} \prod_{j=1}^{p^{\prime}} \frac{\tau(j)}{(q)_{n_j - n_{j-1}}}.
\end{aligned}
\end{equation*}

\end{theorem}
 
As the mirror image of the torus knots $T_{(2, 2t+1)}$ is $\frak{b}(2t+1, 1)$ (cf. \cite{bz, petersen, Tran}), one can check that Theorem \ref{takata} recovers the $q$-hypergeometric expression for $J_N(T_{(2, 2t+1)}; q)$ given in \cite{hikami1, hikami2}. Our interest will be to apply Theorem \ref{takata} to the case of the double twist knots $K_{(m, p)} = \frak{b}(4mp-1, 4mp-2p-1)$ and $K_{(m,-p)}=\frak{b}(4mp+1, 4mp-2p+1)$ (cf. \cite{Tran}). In order to facilitate these computations, we need the following results concerning $\sigma_j$, $i_k$ and $\sigma_{i_k}$. We omit the proofs as they are straightforward generalizations of Lemmas 6--9 in \cite{takata}.
 
 \begin{lemma} \label{l8} For $l=4mp-1$ and $t=4mp-2p-1$, we have \\

\begin{enumerate}
\item[(i)] $\sigma_{j} = \left\{
\begin{array}{ll}
1 & \text{if $j \equiv 1, 2, \dotsc, m \pmod{2m}$} \\
-1 & \text{if $j \equiv 0, m+1, \dotsc, 2m-1 \pmod{2m}.$} \\
\end{array} \right.$ \\

\item[(ii)] To compute $i_k$, apply the following algorithm. Divide the integers from $1$ to $p'$ into $2m-1$ intervals, each of length $p$, and a final interval of length $p -1$. The value of $i_k$ is $2m(p-k) + m$ in the first interval and $1- (2m(p-k) + m)$ in the second.  If $j>1$ is odd, then to obtain the value of $i_k$ in the $j$th interval, add $4mp- 1$ to the formula for $i_k$ in the $(j-2)$th interval. If $j > 2$ is even, then to obtain the value of $i_k$ in the $j$th interval, subtract $4mp - 1$ from the formula for $i_k$ in the $(j-2)$th interval. \\

\item[(iii)] To compute $\sigma_{i_k}$, apply the following algorithm.  Divide the integers from $1$ to $p'$ into $2m-1$ intervals, each of length $p$, and a final interval of length $p - 1$.  The value  of $\sigma_{i_k}$ alternates between $1$ and $-1$ starting with $1$ in the first interval. 
\end{enumerate}

\end{lemma}

\begin{lemma} \label{l9} Let $l=4mp-1$ and $t=4mp-2p-1$.   Then for $1\leq k \leq p'$ and $1\leq j \leq p'-1$ we have \\

\begin{enumerate}

\item[(i)] $\sigma_{i_k} + \sigma_{i_{p^{\prime} + 1 - k}} = \left\{
\begin{array}{ll}
2 & \text{if $k = p, 3p, \dotsc, (2m-1)p$} \\
-2 & \text{if $k = 2p, \dotsc, (2m-2)p$} \\
0 & \text{otherwise.} \\
\end{array} \right.$ \\

\item[(ii)] $\sigma_{j+1} + \sigma_{p^{\prime} + 1 - j} = 0$.
\end{enumerate}
\end{lemma}

\begin{lemma} \label{l6} For $l=4mp + 1$ and $t=4mp - 2p + 1$, we have \\

\begin{enumerate}
\item[(i)] $\sigma_{j} = \left\{
\begin{array}{ll}
1 & \text{if $j \equiv 1, 2, \dotsc, m \pmod{2m}$} \\
-1 & \text{if $j \equiv 0, m+1, \dotsc, 2m-1 \pmod{2m}.$} \\
\end{array} \right.$ \\

\item[(ii)] To compute $i_k$, apply the following algorithm.  Divide the integers from $1$ to $p'$ into $2m$ intervals of length $p$. The value of $i_k$ is $2m(p-k) + m + 1$ in the first interval and $1- (2m(p-k) + m + 1)$ in the second. If $j>1$ is odd, then to obtain the value of $i_k$ in the $j$th interval, add $4mp+1$ to the formula for $i_k$ in the $(j-2)$th interval. If $j > 2$ is even, then to obtain the value of $i_k$ in the $j$th interval, subtract $4mp+1$ from the formula for $i_k$ in the $(j-2)$th interval. \\

\item[(iii)] To compute $\sigma_{i_k}$, apply the following algorithm.  Divide the integers from 1 to $p'$ into $2m$ intervals  of length $p$. The value of $\sigma_{i_k}$ alternates between $-1$ and $1$ starting with $-1$ in the first interval. 
\end{enumerate}

\end{lemma}

\begin{lemma} \label{l7} Let $l=4mp + 1$ and $t=4mp - 2p + 1$.   Then for $1\leq k \leq p'$ and $1\leq j \leq p'-1$ we have \\

\begin{enumerate}

\item[(i)] $\sigma_{i_k} + \sigma_{i_{p^{\prime} + 1 - k}} = 0$ \\

\item[(ii)] $\sigma_{j+1} + \sigma_{p^{\prime} + 1 - j} =  \left\{
\begin{array}{ll}
2 & \text{if $j \equiv 0 \pmod{2m}$} \\
-2 & \text{if $j \equiv m \pmod{2m}$} \\
0 & \text{otherwise.} \\
\end{array} \right.$ \\

\end{enumerate}
\end{lemma}

We now illustrate the computation of $a({\underline{n}})$ and $b_1({\underline{n}}) + b_2({\underline{n}})$ for $l=8p-1$ and $t=6p-1$. The routine evaluation of $X(\underline{n})$
is left to the reader. First, we take $m=2$ in Lemmas \ref{l8} and \ref{l9} to obtain

\begin{equation} \label{step0}
\sigma_{j} = \left\{
\begin{array}{ll}
1 & \text{if $j \equiv 1, 2 \pmod{4}$} \\
-1 & \text{if $j \equiv 0, 3 \pmod{4},$} \\
\end{array} \right. \\
\end{equation}

\begin{equation} \label{step1}
i_k = \left\{
\begin{array}{ll}
4(p-k) + 2 & \text{if $1 \leq k \leq p$,} \\
4(k-p) - 1 & \text{if $p+1 \leq k \leq 2p$,} \\
12p - 4k + 1 & \text{if $2p+1 \leq k \leq 3p$,} \\
 4k - 12p & \text{if $3p+1 \leq k \leq 4p-1$,} \\
\end{array} \right. \\
\end{equation}

\begin{equation} \label{step2}
\sigma_{i_k} = \left\{
\begin{array}{ll}
1 & \text{if $1 \leq k \leq p$,} \\
-1 & \text{if $p+1 \leq k \leq 2p$,} \\
1 & \text{if $2p+1 \leq k \leq 3p$,} \\
 -1 & \text{if $3p+1 \leq k \leq 4p-1,$} \\
\end{array} \right. \\
\end{equation}

\begin{equation} \label{step3}
\sigma_{i_k} + \sigma_{i_{4p-k}} = \left\{
\begin{array}{ll}
2 & \text{if $k = p, 3p$} \\
-2 & \text{if $k = 2p$} \\
0 & \text{otherwise} \\
\end{array} \right. \\
\end{equation}
and
\begin{equation} \label{step4}
\sigma_{j+1} + \sigma_{4p - j} = 0.
\end{equation}

\noindent Applying (\ref{step0}), (\ref{step1}), (\ref{step3}), (\ref{step4}) and reindexing yields

\begin{equation} \label{a}
\begin{aligned}
a({\underline{n}}) & = -\frac{1}{2} \sum_{j=1}^{4p-1} \Biggl( \sum_{k=r^{\prime}(j)}^{4p-1} (\sigma_{i_k} + \sigma_{i_{4p - k}})\Biggr) (n_j - n_{j-1}) - 1 \\
& = -\frac{1}{2} \Biggl[ \sum_{j=1}^{p}  \Biggl( \sum_{k=3p-j+1}^{4p-1} (\sigma_{i_k} + \sigma_{i_{4p -k}})\Biggr) (n_{4j-3} - n_{4j-4}) \\
&  + \sum_{j=1}^{p} \Biggl( \sum_{k=p-j+1}^{4p-1} (\sigma_{i_k} + \sigma_{i_{4p -k}})\Biggr) (n_{4j-2} - n_{4j-3}) + \sum_{j=1}^{p-1} \Biggl( \sum_{k=3p+j}^{4p-1} (\sigma_{i_k} + \sigma_{i_{4p -k}})\Biggr) (n_{4j} - n_{4j-1}) \\
& + \sum_{j=1}^{p} \Biggl( \sum_{k=p+j}^{4p-1} (\sigma_{i_k} + \sigma_{i_{4p -k}})\Biggr) (n_{4j-1} - n_{4j-2}) \Biggr] - 1 \\
& = \sum_{j=1}^{2p-1} (-1)^j n_{2j} - 1.
\end{aligned}
\end{equation}

\noindent By (\ref{step0}) and (\ref{step2}), the second and fifth sums in $b_1({\underline{n}})$ are zero. We then use (\ref{step0})--(\ref{step2}) and reindex to obtain

\begin{equation} \label{thirdpart}
\begin{aligned}
-\sum_{k=p+1}^{4p-1} n_{i_{k} - 1} & = - \Biggl( \sum_{k=p+1}^{2p} n_{i_{k} - 1} + \sum_{k=2p+1}^{3p} n_{i_{k} - 1} + \sum_{k=3p+1}^{4p-1} n_{i_{k} - 1} \Biggr) \\
& = - \Biggl( \sum_{j=1}^{p} (n_{4j-2} + n_{4j-4}) + \sum_{j-1}^{p-1} n_{4j-1} \Biggr),
\end{aligned}
\end{equation}

\begin{equation} \label{fourthpart}
\sum_{k=p+1}^{4p-1} \frac{1+\sigma_{i_k}}{2} n_{i_k} = \sum_{k=2p+1}^{3p} n_{i_k} = \sum_{j=1}^{p} {n_{4j-3}},
\end{equation}

\begin{equation} \label{sixthpart}
\frac{1}{2} \sum_{j=1}^{4p-2} (\sigma_{j+1} - \sigma_j) n_j = \sum_{j=1}^{2p-1} (-1)^j n_{2j}
\end{equation}
and
\begin{equation} \label{b2}
b_2({\underline{n}}) = \sum_{k=p+1}^{3p-1} \frac{1+\sigma_{i_k}}{2} n_{{i_k} - 1} = \sum_{k=2p+1}^{3p-1} n_{{i_k} - 1} = \sum_{j=2}^{p} n_{4j-4}. 
\end{equation}
By (\ref{a})--(\ref{b2}), the sum of $b_2({\underline{n}})$ and the first six terms in $b_1({\underline{n}})$ equals

\begin{equation} \label{sixb1b2}
1 + n_{4p-1}  + \sum_{j=1}^{p} (n_{4j-3} - n_{4j-2} - n_{4j-1}).
\end{equation}

\noindent To compute the seventh term in $b_1({\underline{n}})$, we use (\ref{step1}) and (\ref{step2}) to observe that $k < k^{\prime}$ and $\sigma_{i_k} \neq \sigma_{i_{k^{\prime}}}$ if and only if either $1 \leq k \leq p$ and $p+1 \leq k^{\prime} \leq 2p$ or $1 \leq k \leq p$ and $3p+1 \leq k^{\prime} \leq 4p-1$ or 
$p+1 \leq k \leq 2p$ and $2p+1 \leq k^{\prime} \leq 3p$ or $2p+1 \leq k \leq 3p$ and $3p+1 \leq k^{\prime} \leq 4p-1$.  Also, $\sgn(i_k - i_{k^{\prime}}) = 1$ if and only if $i_k > i_{k^{\prime}}$ and either $i_k = 4p-4k+2$ for $1 \leq k \leq p$ and $i_{k^{\prime}} = 4k^{\prime} - 4p - 1$ for $p+1 \leq k^{\prime} \leq 2p$ or $i_k = 4p-4k+2$ for $1 \leq k \leq p$ and $i_{k^{\prime}} = 4k^{\prime} -12p$ for $3p+1 \leq k^{\prime} \leq 4p-1$ or $i_k = 4k-4p-1$ for $p+1 \leq k \leq 2p$ and $i_{k^{\prime}} = 12p - 4k^{\prime} + 1$ for $2p+1 \leq k^{\prime} \leq 3p$ or $i_k = 12p - 4k + 1$ for $2p+1 \leq k \leq 3p$ and $i_{k^{\prime}} = 4k^{\prime} - 12p$ for $3p+1 \leq k^{\prime} \leq 4p-1$. Taking these cases into account and reindexing, we have

\begin{equation*}
\begin{aligned}
& - \frac{1}{2} \sum_{k=1}^{4p-2} \sum_{k^{\prime} = k+1}^{4p-1} \frac{1 + \sgn(i_k - i_{k^{\prime}})}{2}(\sigma_{i_k} - \sigma_{i_{k^{\prime}}})(n_{i_k} - n_{i_k - 1})(n_{i_{k^{\prime}}} - n_{i_{k^{\prime}} - 1}) \\
& = - \sum_{k=1}^{p} \sum_{k^{\prime} = p+1}^{2p-k} (n_{i_k} - n_{i_k - 1})(n_{i_{k^{\prime}}} - n_{i_{k^{\prime}} - 1}) - \sum_{k=1}^{p} \sum_{k^{\prime}=3p+1}^{4p-k} (n_{i_k} - n_{i_k - 1})(n_{i_{k^{\prime}}} - n_{i_{k^{\prime}} - 1}) \\
& + \sum_{k=p+1}^{2p} \sum_{k^{\prime} = 4p-k+1}^{3p} (n_{i_k} - n_{i_k - 1})(n_{i_{k^{\prime}}} - n_{i_{k^{\prime}} - 1}) - \sum_{k=2p+1}^{3p} \sum_{k^{\prime} = 3p+1}^{6p-k} (n_{i_k} - n_{i_k - 1})(n_{i_{k^{\prime}}} - n_{i_{k^{\prime}} - 1}) \\
& = - \sum_{j=1}^{p} \sum_{j^{\prime}=2}^{j} (n_{4j-2} - n_{4j-3})(n_{4j^{\prime} - 5} - n_{4j^{\prime} - 6}) - \sum_{j=1}^{p} \sum_{j^{\prime}=2}^{j} (n_{4j-2} - n_{4j-3})(n_{4j^{\prime} -4} - n_{4j^{\prime} - 5}) \\
\end{aligned}
\end{equation*}

\begin{equation}  \label{sevenb1}
\begin{aligned}
& + \sum_{j=1}^{p} \sum_{j^{\prime} = 1}^{j} (n_{4j-1} - n_{4j-2})(n_{4j^{\prime} - 3} - n_{4j^{\prime} - 4}) - \sum_{j=1}^{p} \sum_{j^{\prime}=2}^{j} (n_{4j-3} - n_{4j-4})(n_{4j^{\prime}-4} - n_{4j^{\prime} - 5}).
\end{aligned}
\end{equation}

\noindent Finally, using (\ref{step0}) and (\ref{step1}), then reindexing and simplifying gives the eighth term in $b_1(\underline{n})$, 

\begin{equation} \label{lastb1}
\begin{aligned}
& \sum_{j=1}^{4p-1} \sigma_j \Bigl( \sum_{k=1}^{r^{\prime}(j)} (n_{i_k} - n_{i_k - 1}) \Bigr) n_{j-1} \\
& = \sum_{j=1}^{p} \sigma_{4j-3} \Bigl( \sum_{k=1}^{3p-j+1} (n_{i_k} - n_{i_k - 1}) \Bigr) n_{4j-4} +  \sum_{j=1}^{p} \sigma_{4j-2} \Bigl( \sum_{k=1}^{p-j+1} (n_{i_k} - n_{i_k - 1}) \Bigr) n_{4j-3} \\
& +  \sum_{j=1}^{p-1} \sigma_{4j} \Bigl( \sum_{k=1}^{3p+j} (n_{i_k} - n_{i_k - 1}) \Bigr) n_{4j-1} +  \sum_{j=1}^{p} \sigma_{4j-1} \Bigl( \sum_{k=1}^{p+j} (n_{i_k} - n_{i_k - 1}) \Bigr) n_{4j-2} \\
& = \sum_{j=1}^{p} \Biggl( \sum_{k=1}^{p} (n_{4p-4k+2} - n_{4p-4k+1}) + \sum_{k=p+1}^{2p} (n_{4k-4p-1} - n_{4k-4p-2}) \\
& + \sum_{k=2p+1}^{3p-j+1} (n_{12p - 4k + 1} - n_{12p - 4k}) \Biggr) n_{4j-4} + \sum_{j=1}^{p} \Biggl( \sum_{k=1}^{p-j+1} (n_{4p-4k+2} - n_{4p-4k+1}) \Biggr) n_{4j-3} \\
& - \sum_{j=1}^{p-1} \Biggl( \sum_{k=1}^{p} (n_{4p-4k+2} - n_{4p-4k+1}) + \sum_{k=p+1}^{2p} (n_{4k-4p-1} - n_{4k-4p-2}) \\
& + \sum_{k=2p+1}^{3p} (n_{12p - 4k + 1} - n_{12p - 4k}) + \sum_{k=3p+1}^{3p+j} (n_{4k-12p} - n_{4k-12p-1}) \Biggr) n_{4j-1} \\
& - \sum_{j=1}^{p} \Biggl( \sum_{k=1}^{p} (n_{4p-4k+2} - n_{4p-4k+1}) + \sum_{k=p+1}^{p+j} (n_{4k-4p-1} - n_{4k-4p-2}) \Biggr) n_{4j-2} \\
& = \sum_{j=1}^{p} \Biggl( \sum_{j^{\prime}=1}^{j-1} (n_{4j^{\prime} - 1} - n_{4j^{\prime} - 3}) + \sum_{j=j^{\prime}}^{p} (n_{4j^{\prime} - 1} - n_{4j^{\prime} - 4}) \Biggr) n_{4j-4} + \sum_{j=1}^{p} \Biggl( \sum_{j^{\prime} = j}^{p} (n_{4j^{\prime} - 2} - n_{4j^{\prime} - 3} ) \Biggr) n_{4j-3} \\
& - \sum_{j=1}^{p-1} \Biggl( \sum_{j^{\prime} = j+1}^{p} (n_{4j^{\prime} - 1} - n_{4j^{\prime} - 4}) + n_{4j}\Biggr) n_{4j-1} \\
& - \sum_{j=1}^{p} \Biggl( \sum_{j^{\prime}=1}^{j} (n_{4j^{\prime} - 1} - n_{4j^{\prime} - 3}) + \sum_{j^{\prime}=j+1}^{p} (n_{4j^{\prime} - 2} - n_{4j^{\prime} - 3}) \Biggr) n_{4j-2}.
\end{aligned}
\end{equation}

\noindent Thus, combining (\ref{a})--(\ref{lastb1}) implies that $b_1({\underline{n}}) + b_2({\underline{n}})$ equals

\begin{equation*}
\begin{aligned}
& 1 + n_{4p-1}  + \sum_{j=1}^{p} (n_{4j-3} - n_{4j-2} - n_{4j-1}) \\
& - \sum_{j=1}^{p} \sum_{j^{\prime}=2}^{j} (n_{4j-2} - n_{4j-3})(n_{4j^{\prime} - 5} - n_{4j^{\prime} - 6}) - \sum_{j=1}^{p} \sum_{j^{\prime}=2}^{j} (n_{4j-2} - n_{4j-3})(n_{4j^{\prime} -4} - n_{4j^{\prime} - 5}) \\
& + \sum_{j=1}^{p} \sum_{j^{\prime} = 1}^{j} (n_{4j-1} - n_{4j-2})(n_{4j^{\prime} - 3} - n_{4j^{\prime} - 4}) - \sum_{j=1}^{p} \sum_{j^{\prime}=2}^{j} (n_{4j-3} - n_{4j-4})(n_{4j^{\prime}-4} - n_{4j^{\prime} - 5}) \\
& + \sum_{j=1}^{p} \Biggl( \sum_{j^{\prime}=1}^{j-1} (n_{4j^{\prime} - 1} - n_{4j^{\prime} - 3}) + \sum_{j=j^{\prime}}^{p} (n_{4j^{\prime} - 1} - n_{4j^{\prime} - 4}) \Biggr) n_{4j-4} + \sum_{j=1}^{p} \Biggl( \sum_{j^{\prime} = j}^{p} (n_{4j^{\prime} - 2} - n_{4j^{\prime} - 3} ) \Biggr) n_{4j-3} \\
& - \sum_{j=1}^{p-1} \Biggl( \sum_{j^{\prime} = j+1}^{p} (n_{4j^{\prime} - 1} - n_{4j^{\prime} - 4}) + n_{4j}\Biggr) n_{4j-1} \\
& - \sum_{j=1}^{p} \Biggl( \sum_{j^{\prime}=1}^{j} (n_{4j^{\prime} - 1} - n_{4j^{\prime} - 3}) + \sum_{j^{\prime}=j+1}^{p} (n_{4j^{\prime} - 2} - n_{4j^{\prime} - 3}) \Biggr) n_{4j-2}.
\end{aligned}
\end{equation*}

\section{Proofs of Theorems \ref{main1} and \ref{main2}}

\begin{proof}[Proof of Theorem \ref{main1}]
As (\ref{JNK-m-p}) reduces to (\ref{kppos}) when $m=1$ and this case was proven in \cite{takata}, we assume $m \geq 2$. Using Lemmas \ref{l8} and \ref{l9}, one can check that for $l=4mp-1$ and $t=4mp-2p-1$

\begin{equation} \label{a2}
a({\underline{n}}) = \sum_{j=1}^{2p - 1} (-1)^j n_{mj} - 1
\end{equation}

\noindent and $b_1({\underline{n}}) + b_2({\underline{n}})$ equals

\begin{equation*} 
\begin{aligned}
& 1 + n_{2mp - 1} + \sum_{j=1}^{p} \Biggl( \sum_{i=1}^{m-1} n_{2mj - 2m + i} - \sum_{i=m}^{2m-1} n_{2mj - 2m + i} \Biggr) \\
& + \sum_{j=1}^{p} \sum_{j^{\prime}=1}^{j} \sum_{k=1}^{m-1} \sum_{k^{\prime}=1}^{k} 
(n_{2mj - k} - n_{2mj - k - 1})(n_{2mj^{\prime} - 2m + k - k^{\prime} + 1} - n_{2mj^{\prime} - 2m + k - k^{\prime}}) \\
& - \sum_{j=1}^{p} \sum_{j^{\prime}=1}^{j} \sum_{k=1}^{m} \sum_{k^{\prime}=1}^{m-k+1} 
(n_{2mj - m - k + 1} - n_{2mj - m - k})(n_{2mj^{\prime} - 2m - k^{\prime} + 1} - n_{2mj^{\prime} - 2m - k^{\prime}}) \\
& + \sum_{s=1}^{m-1} \sum_{j=1}^{p} \Biggl( \sum_{j^{\prime} = 1}^{j-1} (n_{2mj^{\prime} - s} - n_{2mj^{\prime} - 2m + s}) + \sum_{j^{\prime}=j}^{p} (n_{2mj^{\prime} - s} - n_{2mj^{\prime} - 2m + s-1} ) \Biggr) n_{2mj - 2m + s-1} \\
\end{aligned}
\end{equation*}

\begin{equation}  \label{b1b2c2}  
\begin{aligned}
& + \sum_{j=1}^{p} \Biggl( \sum_{j^{\prime}=j}^{p} (n_{2mj^{\prime} - m} - n_{2mj^{\prime} - m -1}) \Biggr) n_{2mj - m - 1} \\
& - \sum_{j=1}^{p - 1} \Biggl( \sum_{j^{\prime} = j+1}^{p} (n_{2mj^{\prime} - 1} - n_{2mj^{\prime} - 2m} ) + n_{2mj} \Biggr) n_{2mj - 1} \hphantom{xxxxxxxxxxxxxxxxxxxxxxxxx} \\
& - \sum_{s=1}^{m-1} \sum_{j=1}^{p} \Biggl( \sum_{j^{\prime}=1}^{j} (n_{2mj^{\prime} - m + s} - n_{2mj^{\prime} - m - s}) + \sum_{j^{\prime}=j+1}^{p} (n_{2mj^{\prime} - m + s-1} - n_{2mj^{\prime} - m - s}) \Biggr) n_{2mj - m + s - 1}.
\end{aligned}
\end{equation}

\noindent Also, by (\ref{kappa}) and (\ref{tau})

\begin{equation} \label{twist2}
\begin{aligned}
X({\underline{n}}) & = (-1)^{n_{2mp - 1}} q^{-Nn_{2mp - 1}} \frac{(q)_{N-1} (q)_{n_{2mp - 1}}}{(q)_{N - n_{2mp - 1} - 1}} \\
& \times \prod_{j=1}^{p} \frac{(-1)^{n_{2mj - 1} - n_{2mj-m}} q^{\frac{1}{2} \sum\limits_{s=1}^{m} (n_{2mj-2m+s} - n_{2mj-2m+s-1})(n_{2mj-2m+s} - n_{2mj-2m+s-1} + 1)}}{\displaystyle \prod\limits_{s=1}^{2m-1} (q)_{n_{2mj-2m+s} - n_{2mj-2m+s-1}}} \\
& \times \prod_{j=1}^{p- 1} \frac{(-1)^{n_{2mj} - n_{2mj-1}}}{(q)_{n_{2mj} - n_{2mj-1}}}.
\end{aligned}
\end{equation}

\noindent Upon comparing (\ref{spt}) and (\ref{a2})--(\ref{twist2}) with (\ref{JNK-m-p}) and then simplifying, it suffices to prove that for $m \geq 2$
 
\begin{equation*}
\begin{aligned}
& \sum_{j=1}^{p} \sum_{j^{\prime}=1}^{j} \sum_{k=1}^{m-1} \sum_{k^{\prime}=1}^{k} 
(n_{2mj - k} - n_{2mj - k - 1})(n_{2mj^{\prime} - 2m + k - k^{\prime} + 1} - n_{2mj^{\prime} - 2m + k - k^{\prime}}) \\
& - \sum_{j=1}^{p} \sum_{j^{\prime}=1}^{j} \sum_{k=1}^{m} \sum_{k^{\prime}=1}^{m-k+1} 
(n_{2mj - m - k + 1} - n_{2mj - m - k})(n_{2mj^{\prime} - 2m - k^{\prime} + 1} - n_{2mj^{\prime} - 2m - k^{\prime}}) \\
& + \sum_{s=1}^{m-1} \sum_{j=1}^{p} \Biggl( \sum_{j^{\prime} = 1}^{j-1} (n_{2mj^{\prime} - s} - n_{2mj^{\prime} - 2m + s}) + \sum_{j^{\prime}=j}^{p} (n_{2mj^{\prime} - s} - n_{2mj^{\prime} - 2m + s-1} ) \Biggr) n_{2mj - 2m + s-1} \\
& + \sum_{j=1}^{p} \Biggl( \sum_{j^{\prime}=j}^{p} (n_{2mj^{\prime} - m} - n_{2mj^{\prime} - m -1}) \Biggr) n_{2mj - m - 1} \\
& - \sum_{j=1}^{p - 1} \Biggl( \sum_{j^{\prime} = j+1}^{p} (n_{2mj^{\prime} - 1} - n_{2mj^{\prime} - 2m} ) + n_{2mj} \Biggr) n_{2mj - 1} \\
\end{aligned}
\end{equation*}

\begin{equation} \label{lhs}
\begin{aligned}
& - \sum_{s=1}^{m-1} \sum_{j=1}^{p} \Biggl( \sum_{j^{\prime}=1}^{j} (n_{2mj^{\prime} - m + s} - n_{2mj^{\prime} - m - s}) + \sum_{j^{\prime}=j+1}^{p} (n_{2mj^{\prime} - m + s-1} - n_{2mj^{\prime} - m - s}) \Biggr) n_{2mj - m + s - 1} \\
& + \sum_{j=1}^{p} \Biggl[ \binom{n_{2mj-2m+1} + 1}{2} - n_{2mj-2m+1} n_{2mj-2m} + \binom{n_{2mj-m-1}}{2} - n_{2mj-m} n_{2mj-m-1} \\
& + \frac{1}{2} \sum_{s=2}^{m-1} (n_{2mj-2m+s} - n_{2mj-2m+s-1})(n_{2mj-2m+s} - n_{2mj-2m+s-1} + 1) \Biggr]
\end{aligned}
\end{equation}

\noindent equals

\begin{equation} \label{newrhs}
\displaystyle{\sum_{\substack{1 \leq i < j \leq 2mp-1 \\ m \nmid i}} \epsilon_{i,j,m} n_i n_j} - \displaystyle \sum_{i=1}^{2mp-2} n_i n_{i+1}
\end{equation} 

\noindent where $\epsilon_{i,j,m}$ is given by (\ref{epsilondef}). Here, we have used the fact that

\begin{equation} \label{s1}
\sum_{j=1}^{p} \Biggl( \sum_{i=1}^{m-1} n_{2mj - 2m + i} - \sum_{i=m}^{2m-2} n_{2mj - 2m+i} \Biggr) - \sum_{j=1}^{p-1} n_{2mj-1} =  \sum_{i=1}^{2mp-2} \gamma_{i,m} n_i + \sum_{i=1}^{p-1} n_{2mi},
\end{equation}

\noindent where $\gamma_{i,m}$ is given by (\ref{gammadef}), together with the identities

\begin{equation} \label{s2}
\begin{aligned}
& \frac{1}{2} \sum_{j=1}^{p} \sum_{s=1}^{m} (n_{2mj-2m+s} - n_{2mj-2m+s-1})(n_{2mj-2m+s} - n_{2mj-2m+s-1} + 1) \\
& = \sum_{j=1}^{p} \binom{n_{2mj - 2m}}{2} + \binom{n_{2mj-m} + 1}{2} \\
& + \sum_{j=1}^{p} \Biggl[ \binom{n_{2mj-2m+1} + 1}{2} - n_{2mj-2m+1} n_{2mj-2m} + \binom{n_{2mj-m-1}}{2} - n_{2mj-m} n_{2mj-m-1} \\
& + \frac{1}{2} \sum_{s=2}^{m-1} (n_{2mj-2m+s} - n_{2mj-2m+s-1})(n_{2mj-2m+s} - n_{2mj-2m+s-1} + 1) \Biggr] \hphantom{xxxxxxxxxx}
\end{aligned}
\end{equation}

\noindent and

\begin{equation} \label{s3}
\sum_{j=1}^{p} \binom{n_{2mj - 2m}}{2} + \binom{n_{2mj-m} + 1}{2} + \sum_{i=1}^{p-1} n_{2mi}= \sum_{i=1}^{2p-1} \binom{n_{mi} + 1}{2}.
\end{equation}

We now sketch how to proceed from \eqref{lhs} to \eqref{newrhs}.  Let $L_{i}$ denote the $i$th line of \eqref{lhs}. First, note that

\begin{equation} \label{l7l8}
L_7 + L_8 = \sum_{j=0}^{p-1} \sum_{i=1}^{m-1} n_{2mj+i}^2 - \sum_{j=0}^{p-1} \sum_{i=1}^{m} n_{2mj+i-1}n_{2mj+i}.
\end{equation}

\noindent Next, observe that both $L_1$ and $L_2$ simplify as the sum on $k'$ telescopes in each case. Thus,

\begin{equation} \label{newl1hat}
L_1 = \sum_{k=1}^{m-1} \sum_{j=1}^{p} \sum_{j'=1}^{j} (n_{2mj-k} - n_{2mj-k-1})(n_{2mj' - 2m + k} - n_{2mj' - 2m})
\end{equation}

\begin{equation} \label{newl2hat}
L_2 = - \sum_{k=1}^{m} \sum_{j=1}^{p} \sum_{j'=1}^{j} (n_{2mj - m - k +1} - n_{2mj - m- k})(n_{2mj'-2m} - n_{2mj' - 3m + k-1}).
\end{equation}

\noindent Now, splitting $L_1$ into two sums according to the second factor and then using the fact that the resulting second sum telescopes gives

\begin{align} 
L_1 &= \sum_{k=1}^{m-1} \sum_{j=1}^{p} \sum_{j'=1}^{j} (n_{2mj - k} - n_{2mj-k-1}) n_{2mj'-2m+k} + \sum_{j=1}^{p} \sum_{j'=1}^{j} (n_{2mj-m} - n_{2mj-1}) n_{2mj'-2m} \nonumber \\
&= \sum_{k=1}^{m-1} \sum_{j=1}^{p} \sum_{j'=1}^{j} n_{2mj - k}  n_{2mj'-2m+k} - \sum_{k=1}^{m-1} \sum_{j=1}^{p} \sum_{j'=1}^{j} n_{2mj-k-1} n_{2mj'-2m+k}  \nonumber \\
& + \sum_{j=1}^{p} \sum_{j'=1}^{j} n_{2mj-m}n_{2mj'-2m} - \sum_{j=1}^{p} \sum_{j'=1}^{j} n_{2mj-1} n_{2mj'-2m}. \label{newl1}
\end{align}

\noindent A similar simplification for \eqref{newl2hat} yields

\begin{align} 
L_2 &= - \sum_{j=1}^{p} \sum_{j'=1}^{j} (n_{2mj-m} - n_{2mj-2m})  n_{2mj'-2m}  \nonumber \\
&\qquad \qquad \qquad + \sum_{k=1}^{m} \sum_{j=1}^{p} \sum_{j'=1}^{j} (n_{2mj-m-k+1} - n_{2mj-m-k}) n_{2mj'-3m+k-1} \nonumber \\
&=  - \sum_{j=1}^{p} \sum_{j'=1}^{j} n_{2mj-m}n_{2mj'-2m} + \sum_{j=1}^{p} \sum_{j'=1}^{j}n_{2mj-2m}  n_{2mj'-2m} \nonumber \\
& + \sum_{k=1}^{m} \sum_{j=1}^{p} \sum_{j'=1}^{j} n_{2mj-m-k+1}n_{2mj'-3m+k-1} - \sum_{k=1}^{m} \sum_{j=1}^{p} \sum_{j'=1}^{j} n_{2mj-m-k} n_{2mj'-3m+k-1}. 
\label{newl2}
\end{align}

\noindent Observe that the third sum in \eqref{newl1} cancels with the first sum in \eqref{newl2}. Also, if we take $s=1$ in the second triple sum of $L_3$, 
\begin{equation*}
\sum_{s=1}^{m-1} \sum_{j=1}^{p}  \sum_{j^{\prime}=j}^{p} (n_{2mj^{\prime} - s} - n_{2mj^{\prime} - 2m + s-1} ) n_{2mj - 2m + s-1}, 
\end{equation*}
then this cancels with the fourth sum of \eqref{newl1} and the second sum of \eqref{newl2}.  Putting all of this together and expanding sums we have that \eqref{lhs} equals

\begin{equation}  \label{newlhs}
\begin{aligned}
& \sum_{k=1}^{m-1} \sum_{j=1}^{p} \sum_{j'=1}^{j} n_{2mj-k} n_{2mj' - 2m + k} - \sum_{k=1}^{m-1} \sum_{j=1}^{p} \sum_{j'=1}^{j} n_{2mj-k-1} n_{2mj' - 2m + k} \\
& + \sum_{k=1}^{m} \sum_{j=1}^{p} \sum_{j'=1}^{j} n_{2mj-m-k+1}n_{2mj'-3m+k-1} - \sum_{k=1}^{m} \sum_{j=1}^{p} \sum_{j'=1}^{j} n_{2mj-m-k} n_{2mj'-3m+k-1} \\
& + \sum_{s=1}^{m-1} \sum_{j=1}^{p} \sum_{j'=1}^{j-1} n_{2mj'-s}n_{2mj-2m+s-1} - \sum_{s=1}^{m-1} \sum_{j=1}^{p} \sum_{j'=1}^{j-1} n_{2mj'-2m+s} n_{2mj-2m+s-1} \\
& + \sum_{s=2}^{m-1} \sum_{j=1}^{p} \sum_{j'=j}^{p} n_{2mj'-s} n_{2mj-2m + s-1} - \sum_{s=2}^{m-1} \sum_{j=1}^{p} \sum_{j'=j}^{p} n_{2mj'-2m + s-1}  n_{2mj-2m + s-1} \\
& + \sum_{j=1}^{p}  \sum_{j^{\prime}=j}^{p} n_{2mj^{\prime} - m}n_{2mj - m - 1} - \sum_{j=1}^{p}  \sum_{j^{\prime}=j}^{p} n_{2mj^{\prime} - m -1}  n_{2mj - m - 1} \\
& - \sum_{j=1}^{p - 1} \sum_{j^{\prime} = j+1}^{p} n_{2mj^{\prime} - 1}n_{2mj - 1} + \sum_{j=1}^{p - 1} \sum_{j^{\prime} = j+1}^{p} n_{2mj^{\prime} - 2m}n_{2mj - 1} -  \sum_{j=1}^{p - 1}n_{2mj} n_{2mj - 1} \\
& - \sum_{s=1}^{m-1} \sum_{j=1}^{p}  \sum_{j^{\prime}=1}^{j} n_{2mj^{\prime} - m + s}n_{2mj - m + s - 1} +  \sum_{s=1}^{m-1} \sum_{j=1}^{p}  \sum_{j^{\prime}=1}^{j} n_{2mj^{\prime} - m - s}n_{2mj - m + s - 1}  \\
&-  \sum_{s=1}^{m-1} \sum_{j=1}^{p}\sum_{j^{\prime}=j+1}^{p} n_{2mj^{\prime} - m + s-1}n_{2mj - m + s - 1} + \sum_{s=1}^{m-1} \sum_{j=1}^{p}\sum_{j^{\prime}=j+1}^{p}n_{2mj^{\prime} - m - s} n_{2mj - m + s - 1} \\
& + \sum_{j=0}^{p-1} \sum_{i=1}^{m-1} n_{2mj+i}^2 - \sum_{j=0}^{p-1} \sum_{i=1}^{m} n_{2mj+i-1}n_{2mj+i}.
\end{aligned}
\end{equation}

\noindent The second sum on the fifth line of \eqref{newlhs} can now be taken into the second sum of the fourth line, increasing the upper limit of summation there to $m$.  In this sum, we can then exchange $j$ and $j'$ and reindex, giving
\begin{equation*}
-\sum_{s=2}^{m} \sum_{j=1}^{p} \sum_{j'=1}^{j} n_{2mj-2m + s-1}  n_{2mj'-2m + s-1}.
\end{equation*}
Now take out the term $j' = j$ and shift the indices in this term by $j \to j+1$ and $s \to s+1$ to cancel with the first sum on the last line  of \eqref{newlhs}. Finally, in the second line of \eqref{newlhs}, perform the shift $j' \to j'+1$ and start the sum at $j'=1$ (as $j'=0$ gives 0) to obtain

\begin{equation} \label{step}
\sum_{k=1}^{m} \sum_{j=1}^{p} \sum_{j'=1}^{j-1} n_{2mj-m-k+1} n_{2mj'-m+k-1}
- \sum_{k=1}^{m} \sum_{j=1}^{p} \sum_{j'=1}^{j-1} n_{2mj-m-k} n_{2mj'-m+k-1}.
\end{equation}

\noindent We now remove the $k=m$ term from the second sum in \eqref{step} and note that what remains cancels with the second sum in the penultimate line of \eqref{newlhs}.
In total, this yields that \eqref{lhs} equals

\begin{equation} \label{newlhs1}
\begin{aligned}
& \sum_{k=1}^{m-1} \sum_{j=1}^{p} \sum_{j'=1}^{j} n_{2mj-k}n_{2mj' - 2m + k} -\sum_{k=1}^{m-1} \sum_{j=1}^{p} \sum_{j'=1}^{j} n_{2mj-k-1} n_{2mj' - 2m + k} \\
& + \sum_{k=1}^{m} \sum_{j=1}^{p} \sum_{j'=1}^{j-1} n_{2mj-m-k+1} n_{2mj'-m + k -1} - \sum_{j=1}^{p} \sum_{j'=1}^{j-1} n_{2mj-2m} n_{2mj'-1} \\
& + \sum_{s=1}^{m-1} \sum_{j=1}^{p} \sum_{j'=1}^{j-1} n_{2mj'-s}n_{2mj-2m+s-1} - \sum_{s=1}^{m-1} \sum_{j=1}^{p} \sum_{j'=1}^{j-1} n_{2mj'-2m+s} n_{2mj-2m+s-1}  \\
&+ \sum_{s=2}^{m-1} \sum_{j=1}^{p} \sum_{j'=j}^{p} n_{2mj'-s} n_{2mj-2m + s-1}  - \sum_{s=2}^{m} \sum_{j=1}^{p} \sum_{j'=1}^{j-1} n_{2mj-2m+s-1} n_{2mj' - 2m + s-1} \\
& + \sum_{j=1}^{p} \sum_{j'=j}^{p} n_{2mj'-m} n_{2mj-m-1} \\
& - \sum_{j=1}^{p - 1} \sum_{j^{\prime} = j+1}^{p} n_{2mj^{\prime} - 1}n_{2mj - 1} + \sum_{j=1}^{p - 1} \sum_{j^{\prime} = j+1}^{p} n_{2mj^{\prime} - 2m}n_{2mj - 1} -  \sum_{j=1}^{p - 1}n_{2mj} n_{2mj - 1}  \\
& - \sum_{s=1}^{m-1} \sum_{j=1}^{p}  \sum_{j^{\prime}=1}^{j} n_{2mj^{\prime} - m + s}n_{2mj - m + s - 1} +  \sum_{s=1}^{m-1} \sum_{j=1}^{p}  \sum_{j^{\prime}=1}^{j} n_{2mj^{\prime} - m - s}n_{2mj - m + s - 1}  \\
&-  \sum_{s=1}^{m-1} \sum_{j=1}^{p}\sum_{j^{\prime}=j+1}^{p} n_{2mj^{\prime} - m + s-1}n_{2mj - m + s - 1}  \\ 
& - \sum_{j=0}^{p-1} \sum_{i=1}^{m} n_{2mj+i-1}n_{2mj+i}.
\end{aligned}
\end{equation}

\noindent We now simplify further.  The $k=1$ term of the first sum in the second line cancels with the $s=1$ term of the sum in the penultimate line.  The first sum on the fourth line cancels with the second sum of the first line once we remove the $k=m-1$ term. This $k=m-1$ term then cancels with the fifth line.  The first sum in the sixth line is the $s=m$ term of the sum in the penultimate line. The second sum in the sixth line cancels with the second sum in the second line. The last sum in the sixth line is the $i=0$ term in the last line.  Finally, we remove the $j'=j$ term from the first sum in the seventh line and write it in the last line. Thus, \eqref{lhs} equals

\begin{equation*} 
\begin{aligned}
& \sum_{k=1}^{m-1} \sum_{j=1}^{p} \sum_{j'=1}^{j} n_{2mj'-2m+k} n_{2mj-k} + \sum_{k=2}^{m} \sum_{j=1}^{p} \sum_{j'=1}^{j-1} n_{2mj'-m+k-1} n_{2mj-m-k+1} \hphantom{xxxxxxxxxx} \\
\end{aligned}
\end{equation*}

\begin{equation} \label{newlhs2}
\begin{aligned}
& + \sum_{s=1}^{m-1} \sum_{j=1}^{p} \sum_{j'=1}^{j-1} n_{2mj'-s} n_{2mj-2m + s-1} + \sum_{s=1}^{m-1} \sum_{j=1}^{p} \sum_{j'=1}^{j} n_{2mj'-m-s} n_{2mj-m + s-1} \\
& - \sum_{s=1}^{m-1} \sum_{j=1}^{p} \sum_{j'=1}^{j-1} n_{2mj'-2m+s} n_{2mj-2m + s-1} - \sum_{s=1}^{m-1} \sum_{j=1}^{p} \sum_{j'=1}^{j-1} n_{2mj'-m+s} n_{2mj-m+s-1} \\
& - \sum_{s=2}^{m} \sum_{j'=1}^{p} \sum_{j'=1}^{j-1} n_{2mj'-2m + s-1} n_{2mj-2m+s-1} - \sum_{s=2}^{m} \sum_{j=1}^{p-1} \sum_{j'=j+1}^{p} n_{2mj-m+s-1} n_{2mj'-m+s-1} \\
& - \sum_{s=1}^{m-1} \sum_{j=1}^{p} n_{2mj-m+s} n_{2mj-m+s-1} - \sum_{j=0}^{p-1} \sum_{i=0}^{m} n_{2mj+i-1} n_{2mj+i}.
\end{aligned}
\end{equation}

\noindent Now we see that this is equal to \eqref{newrhs} as follows.   The first four lines of \eqref{newlhs2} correspond to the first term in \eqref{newrhs}; namely, the first line of (\ref{newlhs2}) corresponds to $(i,j) \equiv (i,-i) \mod{2m}$, the second line to $(i,j) \equiv (i, -i-1) \mod{2m}$, the third line to $(i,j) \equiv (i,i-1) \mod{2m}$ and the fourth line to $(i,j) \equiv (i,i) \mod{2m}$. Finally, the fifth line of (\ref{newlhs2}) matches the second sum of \eqref{newrhs}. Thus, we have proven that \eqref{lhs} equals \eqref{newrhs}.

\end{proof}

\begin{proof}[Proof of Theorem \ref{main2}]
As (\ref{JNK-mp}) reduces to (\ref{kpneg}) when $m=1$ and this case was proven in \cite{takata}, we assume $m \geq 2$. Using Lemmas \ref{l6} and \ref{l7}, one can check that for $l=4mp+1$ and $t=4mp-2p+1$

\begin{equation} \label{a1}
a({\underline{n}}) = - \sum_{j=1}^{2p - 1} (-1)^j n_{mj}
\end{equation}

\noindent and $b_1({\underline{n}}) + b_2({\underline{n}})$ equals

\begin{equation*}
\begin{aligned}
& \sum_{j=1}^{p} \Biggl[ \sum_{i=0}^{m-1} n_{2mj - 2m + i} - \sum_{i=m+1}^{2m-1} n_{2mj-2m+i} \Biggr] \\
& + \sum_{k=1}^{m} \sum_{j=1}^{p} \sum_{j^{\prime}=1}^{j} (n_{2mj-k+1} - n_{2mj-k})(n_{2mj^{\prime} - 2m + k} - n_{2mj^{\prime} - 2m})  \hphantom{xxxxxxxxxxxxxxxxxxxxxxxxxxxxxxxxxxxx} \\
& - \sum_{s=1}^{m-1} \sum_{j=2}^{p} \sum_{j^{\prime}=2}^{j} (n_{2mj - m - s +1} - n_{2mj-m-s})(n_{2mj^{\prime} - 2m} - n_{2mj^{\prime} - 3m+s}) \\
& + \sum_{j=1}^{p} \Biggl( \sum_{j^{\prime}=j}^{p} (n_{2mj^{\prime}} - n_{2mj^{\prime} - 2m + 1}) + n_{2mj - 2m +1} \Biggr) n_{2mj-2m} \\
& + \sum_{s=1}^{m-1} \sum_{j=1}^{p} \Biggl( \sum_{j^{\prime}=1}^{j} (n_{2mj^{\prime} - s} - n_{2mj^{\prime} - 2m + s}) + \sum_{j^{\prime}=j+1}^{p} (n_{2mj^{\prime} - s} - n_{2mj^{\prime} - 2m + s + 1}) \Biggr) n_{2mj-2m+s} \\
\end{aligned}
\end{equation*}

\begin{equation} \label{b1b2c1}
\begin{aligned}
& - \sum_{j=1}^{p} \Biggl( \sum_{j^{\prime}=j}^{p} (n_{2mj^{\prime} - m + 1} - n_{2mj^{\prime} - m}) \Biggr) n_{2mj-m} \\
& - \sum_{s=1}^{m-1} \sum_{j=1}^{p} \Biggl( \sum_{j^{\prime}=1}^{j-1} (n_{2mj^{\prime} - m + s} - n_{2mj^{\prime} - m - s}) + \sum_{j^{\prime}=j}^{p} (n_{2mj^{\prime} - m + s + 1} - n_{2mj^{\prime} - m - s} ) \Biggr) n_{2mj-m+s}. 
\end{aligned}
\end{equation}

\noindent Also, by (\ref{kappa}) and (\ref{tau})

\begin{equation} \label{twist1}
\begin{aligned}
X({\underline{n}}) & = (-1)^{n_{2mp}} q^{-Nn_{2mp}} \frac{(q)_{N-1} (q)_{n_{2mp}}}{(q)_{N - n_{2mp} - 1}} \\& \times \prod_{j=1}^{p} \frac{(-1)^{n_{2mj} - n_{2mj-m}} q^{\frac{1}{2} \sum\limits_{s=1}^{m} (n_{2mj-2m+s} - n_{2mj-2m+s-1})(n_{2mj-2m+s} - n_{2mj-2m+s-1} + 1)}}{\displaystyle \prod\limits_{s=1}^{2m} (q)_{n_{2mj-2m+s} - n_{2mj-2m+s-1}}}.
\end{aligned}
\end{equation}

\noindent Upon comparing (\ref{spt}) and (\ref{a1})--(\ref{twist1}) with (\ref{JNK-mp}) and then simplifying, it suffices to prove that for $m \geq 2$

\begin{equation*}  
\begin{aligned}
& \sum_{k=1}^{m} \sum_{j=1}^{p} \sum_{j^{\prime}=1}^{j} (n_{2mj-k+1} - n_{2mj-k})(n_{2mj^{\prime} - 2m + k} - n_{2mj^{\prime} - 2m}) \\
& - \sum_{s=1}^{m-1} \sum_{j=2}^{p} \sum_{j^{\prime}=2}^{j} (n_{2mj - m - s +1} - n_{2mj-m-s})(n_{2mj^{\prime} - 2m} - n_{2mj^{\prime} - 3m+s}) \\
& + \sum_{j=1}^{p} \sum_{j^{\prime}=j}^{p} (n_{2mj^{\prime}} - n_{2mj^{\prime} - 2m + 1})n_{2mj-2m} + \sum_{j=1}^{p}n_{2mj - 2m +1}n_{2mj-2m} \\
& + \sum_{s=1}^{m-1} \sum_{j=1}^{p} \Biggl( \sum_{j^{\prime}=1}^{j} (n_{2mj^{\prime} - s} - n_{2mj^{\prime} - 2m + s}) + \sum_{j^{\prime}=j+1}^{p} (n_{2mj^{\prime} - s} - n_{2mj^{\prime} - 2m + s + 1}) \Biggr) n_{2mj-2m+s} \\
& - \sum_{j=1}^{p} \Biggl( \sum_{j^{\prime}=j}^{p} (n_{2mj^{\prime} - m + 1} - n_{2mj^{\prime} - m}) \Biggr) n_{2mj-m} \\
& - \sum_{s=1}^{m-1} \sum_{j=1}^{p} \Biggl( \sum_{j^{\prime}=1}^{j-1} (n_{2mj^{\prime} - m + s} - n_{2mj^{\prime} - m - s}) + \sum_{j^{\prime}=j}^{p} (n_{2mj^{\prime} - m + s + 1} - n_{2mj^{\prime} - m - s} ) \Biggr) n_{2mj-m+s} \\
& + \sum_{j=1}^{p} \Biggl[ \binom{n_{2mj-2m+1} + 1}{2} - n_{2mj-2m+1} n_{2mj-2m} + \binom{n_{2mj-m-1}}{2} - n_{2mj-m} n_{2mj-m-1} \\
\end{aligned}
\end{equation*}

\begin{equation}  \label{lhs2}
\begin{aligned}
& + \frac{1}{2} \sum_{s=2}^{m-1} (n_{2mj-2m+s} - n_{2mj-2m+s-1})(n_{2mj-2m+s} - n_{2mj-2m+s-1} + 1) \Biggr]
\end{aligned}
\end{equation}

\noindent equals 

\begin{equation} \label{Deltasum}
\displaystyle \sum_{\substack{1 \leq i < j \leq 2mp \\ m \nmid i}} \Delta_{i,j,m} n_i n_j
\end{equation}

\noindent where $\Delta_{i,j,m}$ is given by (\ref{Deltadef}). Here, we have used (\ref{s2}), (\ref{s3}) and the fact that

\begin{equation}
\sum_{j=1}^{p} \Biggl[ \sum_{i=0}^{m-1} n_{2mj - 2m + i} - \sum_{i=m+1}^{2m-1} n_{2mj-2m+i} \Biggr] = \sum_{i=1}^{2mp-1} \beta_{i,m} n_i + \sum_{i=1}^{p-1} n_{2mi}
\end{equation}

\noindent where $\beta_{i,m}$ is given by (\ref{betadef}).

We now sketch how to go from \eqref{lhs2} to \eqref{Deltasum}.
Let $\hat{L}_i$ denote the $i$th line of \eqref{lhs2}. We first split $\hat{L}_1$ into two parts according to the second factor and note that the sum on $k$ in the second part telescopes. Thus,

\begin{equation} \label{l1}
\hat{L}_1 = \sum_{k=1}^{m} \sum_{j=1}^{p} \sum_{j'=1}^{j} (n_{2mj-k+1} - n_{2mj-k}) n_{2mj' - 2m + k} + \sum_{j=1}^{p} \sum_{j'=1}^{j} (n_{2mj-m} - n_{2mj}) n_{2mj'-2m}.
\end{equation}

\noindent Similarly, we split $\hat{L}_2$ into two parts according to the second factor and note that the sum on $s$ in the first part telescopes. Thus,

\begin{equation} \label{l2}
\begin{aligned}
\hat{L}_2 = - \sum_{j=1}^{p} \sum_{j'=1}^{j}  (n_{2mj-m} - n_{2mj-2m+1}) & n_{2mj' - 2m} \\
& + \sum_{s=1}^{m-1} \sum_{j=2}^{p} \sum_{j'=2}^{j} (n_{2mj-m-s+1} - n_{2mj-m-s}) n_{2mj-3m+s}.
\end{aligned}
\end{equation}

\noindent Here, we have used the fact that $n_0:=0$. Now, the $k=m$ term of the first sum in (\ref{l1}) cancels with $\hat{L}_5$.  If we combine the double sum of (\ref{l1}) with the double sum in $\hat{L}_3$,  then the resulting sum cancels with the first sum in (\ref{l2}).    Note that $\hat{L}_i = L_i$ for $i=7$ and $8$.  Hence, the single sum in $\hat{L}_3$ cancels with the $i=1$ term of the second sum in (\ref{l7l8}). Putting this together and expanding sums we now have that \eqref{lhs2} equals

\begin{equation} \label{lhs2new}
\begin{aligned}
& \sum_{k=1}^{m-1} \sum_{j=1}^{p} \sum_{j^{\prime}=1}^{j} n_{2mj-k+1}n_{2mj^{\prime} - 2m + k} -\sum_{k=1}^{m-1} \sum_{j=1}^{p} \sum_{j^{\prime}=1}^{j} n_{2mj-k}n_{2mj^{\prime} - 2m + k}
\\
& + \sum_{s=1}^{m-1} \sum_{j=2}^{p} \sum_{j^{\prime}=2}^{j} n_{2mj - m - s +1}n_{2mj^{\prime} - 3m+s} - \sum_{s=1}^{m-1} \sum_{j=2}^{p} \sum_{j^{\prime}=2}^{j}n_{2mj-m-s} n_{2mj^{\prime} - 3m+s} \\
& + \sum_{s=1}^{m-1} \sum_{j=1}^{p}  \sum_{j^{\prime}=1}^{j} n_{2mj^{\prime} - s}n_{2mj-2m+s} - \sum_{s=1}^{m-1} \sum_{j=1}^{p}  \sum_{j^{\prime}=1}^{j} n_{2mj^{\prime} - 2m + s}n_{2mj-2m+s} \\ 
&+ \sum_{s=1}^{m-1} \sum_{j=1}^{p}\sum_{j^{\prime}=j+1}^{p} n_{2mj^{\prime} - s}n_{2mj-2m+s} - \sum_{s=1}^{m-1} \sum_{j=1}^{p}\sum_{j^{\prime}=j+1}^{p}n_{2mj^{\prime} - 2m + s + 1} n_{2mj-2m+s} \\
& - \sum_{s=1}^{m-1} \sum_{j=1}^{p} \sum_{j^{\prime}=1}^{j-1} n_{2mj^{\prime} - m + s}n_{2mj-m+s} + \sum_{s=1}^{m-1} \sum_{j=1}^{p} \sum_{j^{\prime}=1}^{j-1} n_{2mj^{\prime} - m - s}n_{2mj-m+s} \\
&- \sum_{s=1}^{m-1} \sum_{j=1}^{p} \sum_{j^{\prime}=j}^{p} n_{2mj^{\prime} - m + s + 1}n_{2mj-m+s} + \sum_{s=1}^{m-1} \sum_{j=1}^{p} \sum_{j^{\prime}=j}^{p}n_{2mj^{\prime} - m - s}n_{2mj-m+s} \\
& + \sum_{j=0}^{p-1} \sum_{i=1}^{m-1} n_{2mj+i}^2 - \sum_{j=0}^{p-1} \sum_{i=2}^{m} n_{2mj+i-1}n_{2mj+i}.
\end{aligned}
\end{equation}

We combine the $j'=j$ term from the first sum of the third line in (\ref{lhs2new}) with the first sum on the fourth line, and then cancel this with the second sum in the first line.    Next, the $j'=j$ term in the second sum of the third line cancels with the first sum in the last line. Thus, \eqref{lhs2} equals

\begin{equation*} 
\begin{aligned}
& \sum_{k=1}^{m-1} \sum_{j=1}^{p} \sum_{j^{\prime}=1}^{j} n_{2mj-k+1}n_{2mj^{\prime} - 2m + k}
\\
& + \sum_{s=1}^{m-1} \sum_{j=2}^{p} \sum_{j^{\prime}=2}^{j} n_{2mj - m - s +1}n_{2mj^{\prime} - 3m+s} - \sum_{s=1}^{m-1} \sum_{j=2}^{p} \sum_{j^{\prime}=2}^{j}n_{2mj-m-s} n_{2mj^{\prime} - 3m+s} \\
& + \sum_{s=1}^{m-1} \sum_{j=1}^{p}  \sum_{j^{\prime}=1}^{j-1} n_{2mj^{\prime} - s}n_{2mj-2m+s} - \sum_{s=1}^{m-1} \sum_{j=1}^{p}  \sum_{j^{\prime}=1}^{j-1} n_{2mj^{\prime} - 2m + s}n_{2mj-2m+s} \\ 
& - \sum_{s=1}^{m-1} \sum_{j=1}^{p}\sum_{j^{\prime}=j+1}^{p}n_{2mj^{\prime} - 2m + s + 1} n_{2mj-2m+s} \\
& - \sum_{s=1}^{m-1} \sum_{j=1}^{p} \sum_{j^{\prime}=1}^{j-1} n_{2mj^{\prime} - m + s}n_{2mj-m+s} + \sum_{s=1}^{m-1} \sum_{j=1}^{p} \sum_{j^{\prime}=1}^{j-1} n_{2mj^{\prime} - m - s}n_{2mj-m+s} \\
\end{aligned}
\end{equation*}

\begin{equation}  \label{lhs2new2}
\begin{aligned}
&- \sum_{s=1}^{m-1} \sum_{j=1}^{p} \sum_{j^{\prime}=j}^{p} n_{2mj^{\prime} - m + s + 1}n_{2mj-m+s} + \sum_{s=1}^{m-1} \sum_{j=1}^{p} \sum_{j^{\prime}=j}^{p}n_{2mj^{\prime} - m - s}n_{2mj-m+s} \\
& - \sum_{j=0}^{p-1} \sum_{i=2}^{m} n_{2mj+i-1}n_{2mj+i}.
\end{aligned}
\end{equation}

Now, the last line of (\ref{lhs2new2}) is just the $j'=j$ term of the fourth line.   In the second line, perform the shift $j' \to j'+1$ and start the sum at $j=1$. The second sum on this line then cancels with the second sum of penultimate line, except for the $j'=j$ term.  But this term now becomes the $j'=j$ term for the second sum in the fifth line. After simplifying and gathering terms, we have

\begin{equation} \label{lhs2new4}
\begin{aligned}
& \sum_{k=1}^{m-1} \sum_{j=1}^{p} \sum_{j^{\prime}=1}^{j} n_{2mj^{\prime} - 2m + k}n_{2mj-k+1} + \sum_{s=1}^{m-1} \sum_{j=1}^{p} \sum_{j^{\prime}=1}^{j-1} n_{2mj^{\prime} - m+s}n_{2mj - m - s +1} \\
& + \sum_{s=1}^{m-1} \sum_{j=1}^{p}  \sum_{j^{\prime}=1}^{j-1} n_{2mj^{\prime} - s}n_{2mj-2m+s} + \sum_{s=1}^{m-1} \sum_{j=1}^{p} \sum_{j'=1}^{j} n_{2mj^{\prime} - m - s}n_{2mj-m+s} \\
& - \sum_{s=1}^{m-1} \sum_{j=1}^{p} \sum_{j^{\prime}=1}^{j-1} n_{2mj^{\prime} - 2m + s}n_{2mj-2m+s} - \sum_{s=1}^{m-1} \sum_{j=1}^{p} \sum_{j^{\prime}=1}^{j-1} n_{2mj^{\prime} - m + s}n_{2mj-m+s} \\
& - \sum_{s=1}^{m-1} \sum_{j=1}^{p}   \sum_{j^{\prime}=1}^{j} n_{2mj^{\prime}-2m+s} n_{2mj - 2m + s + 1}  - \sum_{s=1}^{m-1} \sum_{j=1}^{p}  \sum_{j^{\prime}=1}^{j} n_{2mj^{\prime}-m+s} n_{2mj - m + s + 1} .
\end{aligned}
\end{equation}

Now we see that this is equal to \eqref{Deltasum} as follows. Namely, the first line of (\ref{lhs2new4}) corresponds to $(i,j) \equiv (i,-i+1)$ mod $2m$, the second line to $(i,j) \equiv (i,-i)$ mod $2m$, the third line to $(i,j) \equiv (i,i)$ mod $2m$ and the fourth line to $(i,j) \equiv (i,i+1)$ mod $2m$. This completes the proof that \eqref{lhs2} is equal to \eqref{Deltasum}.
\end{proof}

\section{Concluding remarks and questions}
The work in this paper may be compared with that of Hikami and the first author in \cite{hikami1,hikami2,hl1}, where one finds generalized Kontsevich-Zagier functions $F_t(q)$ and generalized $U$-functions $U_t(x;q)$ for torus knots $T_{(2,2t+1)}$. In the context of this family of torus knots, we have
\begin{equation} \label{Ft}
F_t(q) := 
  q^{t}
  \sum_{k_t \geq \dots \geq k_1 \geq 0}^\infty
  (q)_{k_t} \, 
  \prod_{i=1}^{t-1} q^{k_i(k_i+1)} 
  \begin{bmatrix}
    k_{i+1} \\
    k_i
  \end{bmatrix} 
\end{equation} 
and 
\begin{equation} \label{Ut}
U_t(x;q) :=
  q^{-t}
  \sum_{k_t  \geq \cdots \geq k_1 \geq 1}
  (-xq)_{k_t-1}(-x^{-1}q)_{k_t-1} \,
  q^{k_t}\prod_{i=1}^{t-1} q^{k_i^2}
  \begin{bmatrix}
    k_{i+1} + k_i - i   + 2\sum_{j=1}^{i-1}k_j \\
    k_{i+1} - k_i
  \end{bmatrix}.
\end{equation}
(Note that $F_1(q) = F_{1,1}(q)$ and $U_1(x;q) = U_{1,1}(x;q)$ since the underlying knot in each case is the trefoil).
While both the torus knot and double twist knot families of functions satisfy the duality in Theorem \ref{fugeneralthm}, much more is known in the case of torus knots.   For example, the functions $F_t(q)$ have explicit quantum modular properties which were given by Hikami \cite{hikami2}. For the case of torus knots $T_{(3, 2^t)}$, see Corollary 4.1 in \cite{go}. As for (\ref{Ut}), it can be written in terms of indefinite ternary theta series \cite{hl1}. It is natural to ask whether the $F_{m,p}(q)$ (and/or $\mathcal{F}_{m,p}(q)$) have quantum modularity or other related properties (e.g., asymptotic expansions near roots of unity), and whether the $U_{m,p}(x;q)$ (and/or the $\mathcal{U}_{m,p}(x;q)$) have any nice representation in terms of indefinite theta series. 

We close with two further questions.   First, for torus knots $T_{(2, 2t+1)}$, the $q$-hypergeometric series expressions for $J_N(T_{(2,2t+1)};q)$ which led to the generalized Kontsevich-Zagier functions $F_t(q)$  were computed in \cite{hikami1, hikami2} using difference equations.  Can one prove  Theorems \ref{main1} and \ref{main2} using this technique?   Second,  both $U_{1,1}(x; q)$ and $F_{1,1}(q)$ are interesting combinatorial generating functions and the coefficients of $F_{1,1}(1-q)$ and $U_{1,1}(1; q)$ satisfy intriguing congruences \cite{ak, as, bopr, garvan, gkr, straub}.  It would be worthwhile to determine if the same is true for $U_{m,p}(x; q)$ and $F_{m,p}(q)$.  

\section*{Acknowledgements}
The authors would like to thank the Mathematisches Forschungsinstitut Oberwolfach for their support as this work began during their stay from March 13-26, 2016 as part of the Research in Pairs program. The second author would like to thank Kazihuro Hikami, Thang L{\^e}, Kate Petersen and Anh Tran for their helpful comments and suggestions.

\end{document}